\documentclass[11pt]{article}
\usepackage[margin=1in]{geometry}

\usepackage{enumerate}
\usepackage{euscript}
\usepackage{graphicx}
\usepackage{amssymb}
\usepackage{amsmath}
\usepackage{amsthm}
\usepackage{dsfont}
\usepackage{color}
\usepackage{breqn}
\newtheorem{assumption}{Assumption}
\newtheorem{defin}{Definition}
\newtheorem{lem}{Lemma}
\newtheorem{rem}{Remark}

\newtheorem{theorem}{Theorem}

\newcommand{\E}{\mathrm{E}}

\newcommand{\R}{\mathrm{Re}}

\def\R{\mathbb{R}}

\def\Z{\mathbb{Z}}

\def\xb{\mathbf{x}}
\def\bx{\boldsymbol{x}}

\def\zbx{\mathbf{X}}

\def\by{\boldsymbol{y}}
\def\bmu{\boldsymbol{\mu}}
\def\r{~}

\begin{document}

\title{\LARGE
Minimizing Regret of Bandit Online Optimization in Unconstrained Action Spaces\let\thefootnote\relax\footnotetext{This research was gratefully funded by the European Union ERC Starting Grant CONENE.}\\
}
\author{Tatiana Tatarenko\thanks{Department of Control Theory and Robotics, TU Darmstadt, Germany} \and Maryam Kamgarpour\thanks{University of British Columbia, Vancouver, Canada}\footnotemark[1]
}

\maketitle

\begin{abstract}
We consider online convex optimization with a zero-order oracle feedback. In particular, the decision maker does not know the explicit representation of the time-varying cost functions, or their gradients. At each time step, she observes the value of the corresponding cost function evaluated at her chosen action (zero-order oracle). The objective is to minimize the regret, that is, the difference between the sum of the costs she accumulates and that of a static optimal action had she known the sequence of cost functions \emph{a priori}.
We present a novel  algorithm to minimize regret in  {unconstrained action spaces}. Our algorithm hinges on a classical idea of one-point estimation of the gradients of the cost functions based on  their observed values. The algorithm is independent of  problem parameters. Letting $T$ denote the number of queries of the zero-order oracle and $n$ the problem dimension, the regret rate achieved is $O(n^{2/3}T^{2/3})$. Moreover, we adapt the presented algorithm to the setting with two-point feedback and demonstrate that the adapted procedure achieves the theoretical lower bound on the regret of $(n^{1/2}T^{1/2})$.
\end{abstract}



\section{Introduction}
\label{sec:intro}
Online convex optimization considers a  time-varying convex objective function with limited information on this function. Due to its  applicability in machine learning over the past decade this problem has been extremely well-studied \cite{shalev2012online, bubeck2012regret}. Past work has considered a plethora of formulations of this problem categorized mainly based on assumptions on the environment (determining the cost) and the decision-maker (equivalently, the learner or algorithm). The environment determines the cost function sequence from a class (linear, strongly convex, convex, smooth). The decision-maker chooses her actions from an action space  (finite, convex, compact, unconstrained) having access to a certain feedback oracle. In particular, the feedback received by the decision-maker can be from  oracle classes of: zero-order, observing only function values at the played action, first-order, observing gradients at the played action, and  full information, observing the function. The measure of performance in this class of problems is \textit{regret}, which is the difference between the accumulated cost for the chosen actions versus the cost corresponding to the optimal static action had the decision-maker had access to the sequence of the cost functions \emph{a priori}.

Deriving lower bounds on the regret achievable and developing algorithms achieving these bounds in uncountable action spaces and with zero-order oracle is an active area of research. The seminal work of \cite{flaxman2005online} extends the online decision-making framework of \cite{zink} from a first-order to a zero-order oracle. The approach introduces randomization in order to derive a one-point estimate of the gradients of the cost functions. Letting $T$ denote the number of queries and $n$ the dimension, this work achieves a regret bound of $O(nT^{3/4})$, applicable to  cost functions with uniformly bounded gradients and a compact convex constraint set. The lower bound for this problem class is established as $\Omega(\sqrt{n^2T})$ \cite{shamir2013complexity}. For the bounded action setting, the follow-the-regularized leader approaches improve the regret bound for Lipschitz losses with Lipschitz gradients: $\tilde{O}(T^{2/3}) $\cite{saha2011improved}, $\tilde{O}(T^{
5/8})$ \cite{dekel2015bandit}, $\tilde{O}(T^{8/13})$ \cite{yang2016optimistic}\footnote{$\tilde{O}$ denotes a potential dependence on $\log(T)$.}. In \cite{agarwal2010optimal}, it was shown that by having access to function values at two query points at each stage (two-point feedback), the optimal regret rate can be achieved in the compact action setting. The dependence of this rate on dimension was improved from quadratic to square root in \cite{shamir2017optimal}. Meanwhile, \cite{hazan2016optimal, bubeck2017kernel} achieve $\tilde{O}(\sqrt{T})$ for compact constraint sets,  at a price of a very high dependence of bounds on problem dimension $n$ and algorithms that lack simplicity of gradient descent type approaches.

The above  results apply to  compact constraint sets and often use the knowledge of the norm of the feasible actions or the existence of self-concordant barriers for the action spaces. The unconstrained action setting has been receiving increasing attention due to its relevance in several learning problems.  The work \cite{mcmahan2014unconstrained} considers unconstrained action spaces but with a first-order oracle and achieves the lower bound of $\Theta(\sqrt{T})$.  Authors in \cite{BachPerchet} address online optimization and consider both constrained and unconstrained action spaces.  They propose two-point and one-point noisy feedback of the cost functions in the unconstrained and constrained action space setting, respectively.  Here,  the regret bounds are refined based on the smoothness degree of the function.  In particular, in the convex Lipschitz losses with Lipschitz gradients they achieve  $O(n^{2/3}T^{2/3})$. However, the points at which the functions are queried differ from the points at which the regret is measured. Thus, they discuss the fact that bandit learning is not considered in their framework. {More recently, \cite{cutkosky2017online,cutkosky2018black} derived a parameter-free
algorithm for  the unconstrained
setting with first-order oracle and achieved optimal regret rates. This line of work was generalized to a noisy first-order oracle, where the noise was sub-exponential \cite{jun2019parameter}. Our work is similar in spirit in that we develop a parameter-free algorithm for unconstrained action spaces. However, in contrast to this past work we consider a zero-order oracle rather than a first-order one and thus, we have to estimate the gradients from function evaluations. }


Our contributions are as follows. We propose a  novel parameter-free  gradient-based algorithm for zero-order  bandit  convex  optimization with {Lipschitz functions and Lipschitz gradients} in unbounded action spaces, with   {$O(n^{2/3}T^{2/3})$} regret rate, consistent with the result in \cite{BachPerchet}, but with  the more restrictive one-point bandit setting.  {As shown in \cite{csaba2016}, with the approach of constructing noisy gradient estimates from one-point feedback the best rate one can reach is $(T^{2/3})$ and this is what we reach}. Our algorithm is arguably simple and does not depend on any constants of the problem.  {The challenge however is on establishing boundedness of the iterates of our algorithm. In particular, since the gradient estimates are obtained from the function evaluations,
we cannot bound the gradient estimates in the unconstrained setting. Hence, our proof technique is considerably different from that of \cite{cutkosky2018black, mcmahan2012no}.  As an additional contribution, we show that assuming one can query the functions at two points, a modified version of our algorithm can achieve the lower bound regret rate of  {$O(n^{1/2}T^{1/2})$}. Our regret rates match the lower bounds of \cite{csaba2016,
shamir2017optimal} in one-point and two-point zero-order feedback respectively, but our proposed algorithms do not require knowing Lipschitz constants of functions or their gradients, bound of a comparator, or the time horizon \textit{a priori}.}

Our approach can be interpreted as stochastic gradient descent on a smoothed version of the cost functions. In particular, by randomizing the query points we obtain a one-point estimate of the gradients of a smoothed version of the cost functions. {This approach is similar  to those based on the Smoothing Lemma \cite{nemirovskii1983problem, flaxman2005online,BachPerchet}. However, we sample the perturbations from a Gaussian distribution, rather than a uniform distribution with compact support. This choice is motivated by the approaches in \cite{Thatha, NesterovGrFree} and the fact that we are not constrained by the requirement of staying inside a feasible set. As such, we choose an appropriate time-varying variance and stepsize parameters for the Gaussian distribution, independent of the problem data, to upper bound the regret. Due to the differences in our setup and approach the regret analysis  of the past work do not apply to our case. Rather, the main step in our convergence analysis is based on finding a  Lyapunov
function to ensure boundedness of the iterates.}

A preliminary and brief version of our work appeared in \cite{tatiana_ECC2018}. We extend our past work in three ways. First, we improve the regret bound from ${O}(T^{3/4})$ to ${O}(T^{2/3})$, based on a modified proof approach inspired from  \cite{BachPerchet}. Second, we update the choice of our variance and step-size parameters to derive the tightest (to the best of our knowledge) dependence of regret on dimension, namely, $n^{2/3} T^{2/3}$ for our setting. Third,  we extend the approach to the setting of two-point zero-order feedback oracle and modify our algorithm to  achieve the regret lower bound of $n^{1/2}T^{1/2}$.

The rest of this paper is organized as follows. In Section\r\ref{sec:problem} we formulate the problem and present some supporting theorems. In Section\r\ref{sec:procedure} we propose the algorithm for the one-point feedback. In Section\r\ref{sec:twopoint} we tighten the regret rates using two-point feedback. We end with concluding remarks in Section\r\ref{sec:conclusion}.

\section{Online Optimization Problem}
\label{sec:problem}
\allowdisplaybreaks
\subsection{Problem formulation}
An unconstrained online convex optimization problem consists of an
infinite sequence $\{c_1, c_2, \ldots \}$, where each $c_t : \R^n \to \R$ is
a convex function.
At each time step $t$, an online convex programming
selects a $\xb_t$ and it receives the  {value of the cost function at the queried point $\hat c_t=c_t(\xb_t)$}.
Efficiency of any online optimization is measured with respect to a regret function defined below.

\begin{defin}\label{def:problem}
Given an algorithm updating $\{\xb_t\}$, and a convex
programming problem $\{c_1 , c_2,\ldots\}$, if $\{\xb_1, \xb_2,\ldots\}$
regret of the algorithm until time $T$ {with respect to a static action $\xb$ is}
\begin{align}\label{eq:regret}
R_{\bx}(T) = \sum_{t=1}^T c_t (\xb_t ) - \sum_{t=1}^T c_t (\bx ).
\end{align}
\end{defin}
The setting is referred to as (oblivious) bandit online optimization.
The goal of the learner is to propose a procedure for the update of $\{\xb_t\}$ such that the average regret function approaches zero, i.e.
$$\varlimsup_{T\to\infty}\frac{R_{\bx}(T)}{T}=\lim_{T\to\infty} \sup_T\frac{R_{\bx}(T)}{T}\le 0, $$
 {with a lowest possible rate dependence on $T$ and on the dimension $n$}.

To address this problem, we need  a set of notations and assumptions. Denote the standard inner product on $\R^n$  by $(\cdot,\cdot)$: $\R^n \times \R^n \to \R$, with the associated norm $\|\bx\|:=\sqrt{(\bx, \bx)}$. {We consider $\{c_1 , c_2,\ldots\}$ belonging to a {class of functions}  satisfying the following assumptions.}

\begin{assumption}\label{assum:Lipsch}
The convex functions $c_t(\bx)$, $t=1, 2, \ldots$, are differentiable and the gradients $\nabla c_t$ are uniformly bounded on $\R^n$.
\end{assumption}

\begin{rem}\label{rem:linbeh}
Since Assumption~\ref{assum:Lipsch} requires uniformly bounded norms of $\nabla c_t$, the functions $c_t(\bx)$, $t=1, 2, \ldots$, grow not faster than a {\color{black}linear} function as $\|\bx\|\to\infty$. Thus, each of these functions is Lipschitz continuous with a Lipschitz constant $l_t$ uniformly bounded by some constant $l$.
\end{rem}

\begin{assumption}\label{assum:Lipsch_grad}
 Each gradient $\nabla c_t$ is Lipschitz continuous on $\R^n$ with some constant $L_t$ and there exists $L$ such that $L_t<L$ for all $t$.
\end{assumption}

 {In the unconstrained action space, to ensure our algorithm iterates do not grow to infinity and  regret  is well-defined we require the following additional assumption.}

\begin{assumption}\label{assum:inftybeh1}
There exists a finite constant $K>0$ such that  for all $t=1,2,\ldots$, $(\bx,\nabla c_t(\bx))>0$, $\forall \|\bx\|^2>K$.
\end{assumption}

\begin{assumption}\label{assum:norm_comp}
The comparator point in the definition of regret \eqref{eq:regret} has finite norm, that is, $\| \bx \| < \infty$.
\end{assumption}

\begin{rem}\label{rem:Assum_inf}
%
Note that the bounds on the Lipschitz constants of the function or their gradients or the bound $K$ and $\| \bx  \|$ are not  needed by the algorithm presented further in the paper and are only used in the analysis.
\end{rem}

Let us discuss the assumptions above in the context of past work. The assumptions on Lipschitz costs with bounded gradients is a regularity condition needed to bound regret and is employed  in all past work dealing with regret in online optimization. In the unconstrained action spaces with zero-order oracle, we use Assumption\r\ref{assum:Lipsch_grad} to bound the difference between the so-called smooth costs and original costs (see proof of our Lemma\r\ref{lem:mixedstr_cost} Part 2 and the use of this Lemma in proof of regret bounds). This assumption is also employed in \cite{saha2011improved}, \cite{dekel2015bandit}, \cite{yang2016optimistic} to improve upon existing regret bounds. Only few works deal with regret in unconstrained action setting.  {In this setting, most works such as  \cite{mcmahan2014unconstrained, cutkosky2017online, cutkosky2018black, jun2019parameter} evaluate regret with respect to a fixed action with a bounded norm.  These works do not perform a one-point estimation of gradients based on zero-order feedback and thus, do not run into difficulty of bounding algorithm iterates or the gradient estimates.  Our Assumption~\ref{assum:inftybeh1} ensures regret is well-defined in the unconstrained case with respect to any comparator with a  finite norm and furthermore, ensures boundedness of the iterates of the algorithm despite having no bound on the estimate of the gradients.  Furthermore, while some past works on unconstrained regret minimization require knowledge of the Lipschitz constants or the comparator's norm for the algorithm design, our work is similar to \cite{jun2019parameter} in that our algorithm  does not require any parameters of the problem.

Finally, to provide further intuition into Assumption~\ref{assum:inftybeh1}, note that if the functions $c_t$ are coercive, then this assumption is equivalent to minima of these functions being attained within a ball of radius $K$.  Indeed, if  $c_t$ is coercive, namely $\lim_{\|\bx\|\to\infty}c_t(\bx)= \infty$, there exists $K$ such that $c_t(\bx)>c_t(\boldsymbol 0)$ for all $\bx$ such that $\|\bx\|^2>K$. Thus,
\[c_t(\boldsymbol 0)\ge c_t(\bx) + (\nabla c_t(\bx), \boldsymbol 0 - \bx)>c_t(\boldsymbol 0) -(\nabla c_t(\bx),  \bx),\]
for any such $\bx$, where the first inequality is due to convexity. The above implies $(\nabla c_t(\bx),  \bx) > 0$ on the set $\{\bx: \, \|\bx\|^2>K\}$. On the other hand, if $(\bx,\nabla c_t(\bx))>0$, $\forall \|\bx\|^2>K$, there is no minima of $c_t$ on the set $\{\bx: \, \|\bx\|^2>K\}$ (due to the first-order optimality condition). Hence, the continuous function $c_t$ attains its minimum on the compact set $\{\bx: \, \|\bx\|^2\le K\}$. If the functions $c_t$ are not coercive, then the optima may  be not uniformly bounded. Despite this, Assumption~\ref{assum:inftybeh1} enables us to bound the norm of the iterates of the algorithm (see Lemma~\ref{lem:lem1} below) and, ultimately, to be able to derive a regret rate that holds with respect to any comparator with a bounded norm, please see statements of Theorems~\ref{th:main} and \ref{th:main2P} and proofs.  Geometrically, this assumption ensures that as iterates get large, the (approximated) gradient step moves in a direction that the iterates' norm would be reduced.
}
\newpage

\section{Proposed Online Optimization Algorithm}\label{sec:procedure}

\subsection{One-point estimate of the gradients}\label{sec:unconstrained}
 {The idea of the algorithm is  simple. We mimic gradient descent as per the online learning approach in the seminal work of \cite{zink}. However, since in zero-order oracle gradients are unavailable, we use the widely employed idea of one-point estimate of the gradients based on the function values at queried points.  Thus, by perturbing the query points with a zero-mean noise and appropriately scaling the queried function values, we obtain an unbiased estimate of the gradients. In contrast to methods that sample the noise from a uniform distribution, in the unconstrained setting we use  Gaussian distribution. This enables our algorithm to be parameter free. We are now ready to formally state the algorithm.}

The proposed algorithm is as follows. At each time $t$ the decision maker chooses the vector $\xb_t$ according to the $n$-dimensional normal distribution $\EuScript N(\bmu_t,\sigma_t)$, $\bmu_t = (\mu_t^1,\ldots,\mu_t^n)$, meaning that the coordinates $x_t^1,\ldots,x_t^n$ of the random vector $\xb_t$ are independently distributed with the mean values $\mu_t^1\ldots,\mu_t^n$ and the variance $\sigma_t$. The mean value initial condition $\bmu_0$ can be chosen arbitrarily. The iterate $\bmu_t$ is updated using the observed value of the cost function $\hat{c}_t = c_t(\xb_t)$ as follows
\begin{align}\label{eq:alg}
\xb_t &\sim\EuScript N(\bmu_t,\sigma_t),\\
\nonumber
 \bmu_{t+1} &= \bmu_t - \alpha_t \hat c_t\frac{\xb_t - \bmu_t}{\sigma^2_t}.
\end{align}
Before presenting the regret bound of our proposed   algorithm, we provide some insights. In particular, we show that our algorithm can be interpreted as a stochastic optimization procedure.
First, let  the $\sigma$-algebra  $\EuScript F_t = \sigma(\{\bmu_k\}_{k\le t})$ be generated by the random variables $\{\bmu_k\}_{k\le t}$. The corresponding filtration $\{\EuScript F_t\}_t$ is a standard object used in analysis of stochastic processes (see, for example, \cite{NH}, basic properties of Markov processes after Definition~2.1.2). Further, let $\E_{t}\{\cdot\}$  denote the conditional expectation of a random variable with respect to the $\sigma$-algebra  $\EuScript F_t $, i.e. $\E_{t}\{\cdot\} = \E \{\cdot | \EuScript F_t\}$.


Let us introduce the smoothed cost as
\begin{align}\label{eq:mixedstr_cost}
\tilde c_t(\bmu_t) = \int_{\R^n} c_t(\bx)p(\bmu_t,\sigma_t,\bx)d\bx,
\end{align}
where $\bx=(x^1,\ldots,x^n)$ and $$p(\bmu_t,\sigma_t,\bx)=\frac{1}{(\sqrt{2\pi}\sigma_t)^{n}}\exp\left\{-\sum_{k=1}^{n}\frac{(x^k-\mu^k_t)^2}{2\sigma_t^2}\right\},$$
is the density of $\EuScript N(\bmu_t,\sigma_t)$. Thus, $\tilde c_t(\bmu_t)$  is the expectation $\E_{t}\{c_t(\xb_t)\}$ of the random variable $c_t(\xb_t)$, given that $\xb_t$ has the normal distribution $\EuScript N(\bmu_t,\sigma_t)$. The iteration \eqref{eq:alg} can then be rewritten as a stochastic gradient descent on the the function $\tilde c_t(\bmu_t)$:
\begin{align}
\label{eq:alg1}
 \bmu_{t+1} = \bmu_t - \alpha_t\nabla\tilde c_t(\bmu_t)+\alpha_t \xi_t(\xb_t,\bmu_t, \sigma_t),
\end{align}
where
$$\xi_t(\xb_t,\bmu_t, \sigma_t) = \nabla\tilde c_t(\bmu_t)-\hat c_t\frac{\xb_t - \bmu_t}{\sigma^2_t}.$$
\begin{lem}\label{lem:zeta_martingale}
Under Assumption\r\ref{assum:Lipsch}, the following two equalities hold:
\begin{align}\label{eq:martdiff}
\E_{t}\xi_t(\xb_t,\bmu_t, \sigma_t) &= \nabla\tilde c_t(\bmu_t)-\E_{t}\{\hat c_t\frac{\xb_t - \bmu_t}{\sigma^2_t}\} = 0.\\
\label{eq:martdiff1}
\nabla\tilde c_t(\bmu_t) &= \int_{\R^n}\nabla c_t(\bx)p(\bmu_t,\sigma_t,\bx)d\bx.
\end{align}
\end{lem}
\begin{proof}
See Appendix\r\ref{app:integral_proofs}.
\end{proof}

The interpretation of the lemma above is that in  Procedure \eqref{eq:alg}, we use unbiased random estimations of the gradients of a smooth version of the cost functions
\begin{align*}
\E_{t}\left\{\hat c_t\frac{\xb_t - \bmu_t}{\sigma^2_t}\right\}=\nabla \int_{\R^n} c_t(\bx)p(\bmu_t,\sigma_t,\bx)d\bx.
\end{align*}

We derive some further properties of the terms in Procedure\r\eqref{eq:alg} that will be used in the convergence analysis. Their full proofs can be found in Appendix\r\ref{app:integral_proofs}.

\begin{lem}
\label{lem:mixedstr_cost}
Under Assumptions~\ref{assum:Lipsch}, \ref{assum:Lipsch_grad}, and\r\ref{assum:inftybeh1}, the functions $\tilde c_t(\bmu_t)$, $t=1,\ldots$ enjoy the following properties:
\begin{enumerate}
\item $\tilde c_t(\bmu_t)$, $t=1,\ldots$, are convex on $\R^n$ and their gradients $\nabla \tilde c_t(\bmu_t)$ are uniformly bounded.
\item For $t=1, \ldots$
\begin{align}\label{eq:deviation}
|c_t(\bmu_t)-\tilde c_t(\bmu_t)|\le \frac{nL\sigma_t^2}{2}, \; \forall \bmu_t \in \R^n.
\end{align}
\item There exists a finite constant $\tilde K>0$ such that $(\bmu,\nabla \tilde c_t(\bmu))>0$ for $\|\bmu\|^2>\tilde K$, $\forall t$.
\end{enumerate}
\end{lem}


\begin{lem}\label{lem:234_moments} Under Assumption~\ref{assum:Lipsch} for any $t$
 \begin{align*}
  \E_{t}\{\|\xi_t(\xb_t,\bmu_t,\sigma_t)\|^2\}\le \frac{f_1(\bmu_t, \sigma_t)}{\sigma_t^2}, \\
  \E_{t}\{\|\xi_t(\xb_t,\bmu_t,\sigma_t)\|^3\}\le \frac{f_2(\bmu_t, \sigma_t)}{\sigma_t^3},\\
  \E_{t}\{\|\xi_t(\xb_t,\bmu_t,\sigma_t)\|^4\}\le \frac{f_3(\bmu_t, \sigma_t)}{\sigma_t^4},
 \end{align*}
 where $f_1(\bmu_t, \sigma_t)$, $f_2(\bmu_t, \sigma_t)$, and $f_3(\bmu_t, \sigma_t)$ are polynomials of $\sigma_t$ and are second, third, and fourth order polynomials of $\mu^i_t$, $i\in[n]$, respectively. Moreover,
  \begin{align}\label{eq:dim_depend}
  \E_{t}\{\|\xi_t(\xb_t,\bmu_t,\sigma_t)\|^2\}= O\left(\frac{nl^2\|\bmu_t\|^2}{\sigma_t^2} + 1\right).
  \end{align}
\end{lem}
 Equipped with the  parallels of the proposed Algorithm \eqref{eq:alg} with the stochastic gradient procedure in \eqref{eq:alg1} and its properties formulated in the lemmas above, we are ready to present the regret bounds for the procedure.

\subsection{Derivation of regret bounds}\label{sec:convergence}

\begin{theorem}\label{th:main}
Let \eqref{eq:alg} define the optimization algorithm for the unconstrained online convex optimization problem $(\R^n,\{c_1,c_2,\ldots\})$. Choose the step-sizes and variances according to   {$\{\alpha_t = \frac{1}{(nt)^a}\}$, $\{\sigma_t= \frac{1}{(nt)^b}\}$, where $0<a<1$, $b>0$, $2a-2b>1$. Then,  under Assumptions\r\ref{assum:Lipsch}-\ref{assum:inftybeh1},
the regret of  algorithm \eqref{eq:alg} estimated with respect to any ${\bx }$ with bounded norm satisfies
\begin{align*}
  \E\left\{\frac{R_{\bx}(T)}{T}\right\} =(M^2+\|\bx\|^2)O\left(\frac{n^a}{T^{1-a}}\right) + l^2O\left(\frac{n^{1-a+2b}}{T^{a-2b}}\right)+ LO\left(\frac{n^{1-2b}}{T^{2b}}\right)
\end{align*}
In particular, $$ \varlimsup_{T\to\infty} \E \left\{\frac{R_{\bx}(T)}{T}\right\}\le 0,$$
for any $a$, $b$ satisfying the assumption above. For the optimal choice of $a \to\frac{2}{3} $, $b\to\frac{1}{6}$, the expected regret of the query points satisfies the bound:
\begin{align*}
\E\left\{\frac{R_{\bx}(T)}{T}\right\} = (M^2+\|\bx\|^2+L+l^2)O\left(\frac{n^{\frac{2}{3}}}{T^{\frac{1}{3}}}\right).
\end{align*}}

\end{theorem}


To prove the main theorem above, we first show that under the conditions of this theorem, the mean values $\{\bmu_t\}$ stay almost surely bounded during the process \eqref{eq:alg}.

\begin{lem}\label{lem:lem1} {Consider the optimization algorithm  \eqref{eq:alg} with step-size $\{\alpha_t = O\left(\frac{1}{t^a}\right)\}$ and variance $\{\sigma_t= O\left(\frac{1}{t^b}\right)\}$, where $0<a<1$, $b>0$, $2a-2b>1$. There exists a finite constant $M$ such that
$$\Pr\{\|\bmu_t\|\le M, t=1,2,\ldots\,|\,\|\bmu_{0}\|<\infty\}=1$$
for all $\{c_1,c_2,\ldots\}$ satisfying Assumptions\r\ref{assum:Lipsch} and\r\ref{assum:inftybeh1}.
In words, given any $\bmu_0$ with a bounded norm, $\|\bmu_t\|$ is bounded almost surely by a constant $M$, uniformly with respect to time $t$ and the sequences of cost functions $\{c_1,c_2,\ldots\}$.}
\end{lem}
\begin{proof}
First, we notice that the conditions on the sequences $\{\alpha_t\}$, $\{\sigma_t\}$ imply that
\begin{align}\label{eq:seq}
 \sum_{t=1}^{\infty}\alpha_t=\infty, \quad  \sum_{t=1}^{\infty}\frac{\alpha_t^2}{\sigma_t^2}<\infty.
\end{align}
Let us consider the function $V(\bmu) = W(\|\bmu\|^2)$, where $W:\R\to\R$ is defined as follows:
\begin{align}\label{eq:funV}
W(x) = \begin{cases}
        0, &\mbox{ if } x<\tilde K,\\
        (x-\tilde K)^2, &\mbox{ if } x\ge \tilde K,
       \end{cases}
\end{align}
and $\tilde K$ is the constant from Lemma\r\ref{lem:mixedstr_cost}. Let $\EuScript L$ denote the generating operator of the Markov process $\{\bmu_t\}$. Recall that $\EuScript LV (\bmu) = \E\{V(\bmu_{t+1})|\bmu_t=\bmu\} - V(\bmu)$. Our goal is to apply a result on boundedness of the discrete-time Markov processes, based on the generating operator properties applied to the function $V$. This result is provided in \cite{NH}, Theorem\r2.5.2. (For the ease of  reviewers, we  provided the statement of this theorem in Appendix A).

The function $W$ fulfills the following property
\begin{align}
 W(y)-W(x)\le W'(x)(y-x) + (y-x)^2,
\end{align}
Thus, taking this inequality into account, we obtain
\begin{align}
 \label{eq:propW}
 V(\bmu_{t+1}) - V(\bmu_t)\le &W'(\|\bmu_t\|^2)(\|\bmu_{t+1}\|^2 - \|\bmu_t\|^2) + (\|\bmu_{t+1}\|^2 - \|\bmu_t\|^2)^2.
\end{align}

According to \eqref{eq:alg1} the  norm of $\bmu_t$ evolves as
\begin{align}\label{eq:firstest}
 \|\bmu_{t+1}\|^2 = &\|\bmu_t - \alpha_t\nabla \tilde c_t(\bmu_t) +\alpha_t \xi_t(\xb_t,\bmu_t, \sigma_t)\|^2\cr
 =&\|\bmu_t \|^2 + \alpha^2_t(\|\nabla \tilde c_t(\bmu_t)\|^2 + \|\xi_t(\xb_t,\bmu_t, \sigma_t)\|^2) \cr
 &+2\alpha_t(-\nabla \tilde c_t(\bmu_t)+ \xi_t(\xb_t,\bmu_t, \sigma_t),\bmu_t )-2\alpha_t^2(\nabla \tilde c_t(\bmu_t), \xi_t(\xb_t,\bmu_t, \sigma_t)).
\end{align}
Hence, taking into account \eqref{eq:martdiff}, we obtain
\begin{align}
 \E_{t}\|\bmu_{t+1}\|^2 =& \|\bmu_t \|^2+ \alpha^2_t(\|\nabla  \tilde c_t(\bmu_t)\|^2 + \E_{t}\|\xi_t(\xb_t,\bmu_t, \sigma_t)\|^2) \cr
 &-2\alpha_t(\nabla  \tilde c_t(\bmu_t),\bmu_t )
\label{eq:est1}
\end{align}

We proceed with estimation of the term $\E_{t}\|\xi_t(\xb_t,\bmu_t, \sigma_t)\|^2$.
Due to Lemma\r\ref{lem:234_moments} for some quadratic function of $\sigma_t$ and $\mu^i_t$, $i\in[n]$, denoted by $f_1(\bmu_t, \sigma_t)$, we have
 \begin{align}\label{eq:Rineq}
  \E_{t}\{\|\xi_t(\xb_t,\bmu_t,\sigma_t)\|^2\}\le \frac{f_1(\bmu_t, \sigma_t)}{\sigma_t^2}.
 \end{align}

From Assumptions\r\ref{assum:Lipsch} and\r\ref{assum:inftybeh1}, \eqref{eq:propW}-\eqref{eq:Rineq}, and   {the fact that $\E\{\cdot|\bmu_t=\bmu\} =\E\{\E_{t}\{\cdot\}|\bmu_t=\bmu\}$}
  \begin{align}
  \nonumber
    \label{eq:part1}
  &\EuScript LV (\bmu) = \E\{V(\bmu_{t+1})|\bmu_t=\bmu\} - V(\bmu)\\
  &\le(\E\{\E_{t}\{\|\bmu_{t+1}\|^2\}|\bmu_t=\bmu\} - \|\bmu\|^2)W'(\|\bmu\|^2) + \E\{\E_{t}\{(\|\bmu_{t+1}\|^2 - \|\bmu_t\|^2)^2\} |\bmu_t=\bmu\}\cr
  &\le -2\alpha_t(\nabla \tilde c_t(\bmu),\bmu)W'(\|\bmu\|^2)\cr
  &\quad + W'(\|\bmu\|^2) \alpha^2_t(\|\nabla \tilde c_t(\bmu)\|^2 + \E\{\E_{t}\{\|\xi_t(\xb_t,\bmu_t, \sigma_t)\|^2\}|\bmu_t=\bmu\})\cr
  &\qquad+\E\{\E_{t}\{(\|\bmu_{t+1}\|^2 - \|\bmu_t\|^2)^2\}|\bmu_t=\bmu\}\cr
  &\le-2\alpha_t(\nabla \tilde c_t(\bmu),\bmu)W'(\|\bmu\|^2)+ g_1(t)(1+V(\bmu)) \cr
  &\qquad+\E\{\E_{t}\{(\|\bmu_{t+1}\|^2 - \|\bmu_t\|^2)^2\}|\bmu_t=\bmu\},
 \end{align}
  { where
\begin{align*}
  g_1(t) &= \frac{W'(\|\bmu\|^2)\alpha^2_t(\|\nabla \tilde c_t(\bmu)\|^2 + \E_{t}\|\xi_t(\xb_t,\bmu, \sigma_t)\|^2)}{1+V(\bmu)}\cr
   & =\alpha_t^2 \frac{W'(\|\bmu\|^2)\|\nabla \tilde c_t(\bmu)\|^2}{1+V(\bmu)} +  \frac{\alpha_t^2}{\sigma_t^2} \frac{f_1(\bmu, \sigma_t)}{1+V(\bmu)}.
\end{align*}
From the definition of $W$, $V$, and $f_1$ and the fact that $\nabla \tilde c_t(\bmu)$ is bounded by Lemma~\ref{lem:mixedstr_cost} Part 1, we conclude that $g_1(t) = O\left(\frac{\alpha_t^2}{\sigma_t^2}\right)$}. Thus, according to the condition in \eqref{eq:seq}, $\sum_{t=1}^{\infty}g_1(t)<\infty$.
Finally, we estimate the term $\E\{\E_{t}\{(\|\bmu_{t+1}\|^2 - \|\bmu_t\|^2)^2\}|\bmu_t=\bmu\}$. According to \eqref{eq:firstest}
 \begin{align*}
 &(\|\bmu_{t+1}\|^2 - \|\bmu_t\|^2)^2= [\alpha^2_t(\|\nabla \tilde c_t(\bmu_t)\|^2  + \|\xi_t(\xb_t,\bmu_t, \sigma_t)\|^2) \cr
 &\qquad+2\alpha_t(-\nabla \tilde c_t(\bmu_t) + \xi_t(\xb_t,\bmu_t, \sigma_t),\bmu_t )-2\alpha_t^2(\nabla  \tilde c_t(\bmu_t), \xi_t(\xb_t,\bmu_t, \sigma_t))]^2.
 \end{align*}
   {Hence,
\begin{align}
\label{eq:secondterm}
&(\|\bmu_{t+1}\|^2 - \|\bmu_t\|^2)^2\le[\alpha^2_t(\|\nabla \tilde c_t(\bmu_t)\|^2  + \|\xi_t(\xb_t,\bmu_t, \sigma_t)\|^2) \\
\nonumber
 &+2\alpha_t\|\bmu_t\|(\|\nabla \tilde c_t(\bmu_t)\|+ \|\xi_t(\xb_t,\bmu_t, \sigma_t)\|)+2\alpha_t^2\|\nabla \tilde c_t(\bmu_t)\|\| \xi_t(\xb_t,\bmu_t, \sigma_t)\|]^2\\
 \nonumber
 &= [\alpha^2_t \|\xi_t(\xb_t,\bmu_t, \sigma_t)\|^2 + O(\alpha_t)\|\xi_t(\xb_t,\bmu_t, \sigma_t)\|(\|\bmu_t\|+1)+ O(\alpha_t)(\|\bmu_t\|+1)]^2\\
 \nonumber
 & = \alpha^4_t \|\xi_t(\xb_t,\bmu_t, \sigma_t)\|^4 +  O(\alpha^2_t)(\|\xi_t(\xb_t,\bmu_t, \sigma_t)\|^2+1)(\|\bmu_t\|+1)^2\\
 \nonumber
 &+  O(\alpha^3_t)(\|\xi_t(\xb_t,\bmu_t, \sigma_t)\|^2+\|\xi_t(\xb_t,\bmu_t, \sigma_t)\|^3)(\|\bmu_t\|+1),
 \end{align}
 where in the first equality above, we used Lemma~\ref{lem:mixedstr_cost} Part 2 on boundedness of $\tilde c_t(\bmu_t)$.
Taking conditional expectation $\E_{t}\{\cdot\}$ of the both sides, we see that we need to bound $\E_{t}\|\xi_t(\xb_t,\bmu_t, \sigma_t)\|^3$ and $\E_{t}\|\xi_t(\xb_t,\bmu_t, \sigma_t)\|^4$. Using  Lemma\r\ref{lem:234_moments} we obtain
  \begin{align}\label{eq:Rineq1}
  \E_{t}\{\|\xi_t(\xb_t,\bmu_t,\sigma_t)\|^3\}\le \frac{f_2(\bmu_t, \sigma_t)}{\sigma_t^3},\\
\label{eq:Rineq2}
  \E_{t}\{\|\xi_t(\xb_t,\bmu_t,\sigma_t)\|^4\}\le \frac{f_3(\bmu_t, \sigma_t)}{\sigma_t^4},
 \end{align}
 where $f_2(\bmu_t, \sigma_t)$ and $f_3(\bmu_t, \sigma_t)$ are third and fourth order polynomials of $\sigma_t$ and $\mu^i_t$, $i\in[n]$, respectively.
 Hence,  from Lemma~\ref{lem:234_moments} and definition of the function $V(\bmu) = O(\|\bmu\|^4)$ we get for large enough $\bmu$
\begin{align}
\nonumber
\E_{t}\{(\|\bmu_{t+1}\|^2 - \|\bmu_t\|^2)^2|\bmu_t = \bmu\} =  & O\left(\frac{\alpha^4_t}{\sigma_t^4}\|\bmu\|^4\right) + O\left(\frac{\alpha^2_t}{\sigma_t^2}(1+\|\bmu\|^2)\right)\\
\label{eq:part2}
&+O\left(\frac{\alpha^3_t}{\sigma_t^3}(\|\bmu\|^4+1)\right) = g_2(t)(1+V(\bmu)),
\end{align}
 where $g_2(t) = O\left(\frac{\alpha_t^2}{\sigma_t^2}\right)$}. Thus, due to condition for step-sizes and variances,
it follows from  \eqref{eq:seq} that $\sum_{t=1}^{\infty}g_2(t)<\infty$. Thus, from \eqref{eq:part1} and \eqref{eq:part2} we obtain
 \begin{align}
 \label{eq:V_onepoint}
 \EuScript LV (\bmu)\le &-2\alpha_t(\nabla \tilde c_t(\bmu),\bmu)W'(\|\bmu\|^2) + (g_1(t)+g_2(t))(1+V(\bmu)).
 \end{align}
  {Due to Lemma 2 and definition of the function $W$ (see \eqref{eq:funV}), $(\nabla \tilde c_t(\bmu),\bmu)W'(\|\bmu\|^2) \ge  0$ for any $t = 1, 2, \ldots,$ $\bmu \in\R^n$ and $(\nabla \tilde c_t(\bmu),\bmu)W'(\|\bmu\|^2) > 0$ for any $t = 1, 2, \ldots,$
and $\bmu$ such that $\|\bmu\|> ˜K$ . Moreover, according to \eqref{eq:seq}, $\sum_{t=1}^{\infty} \alpha_t = \infty$.  Thus,
Theorem~2.7.1 from [13] (see Appendix~\ref{app:support}) implies that $\lim_{t\to\infty}\mbox{dist}(\bmu_t,B_{\tilde{K}})=0$ almost surely, where
$B_{\tilde{K}}=\{\bmu: \|\bmu\|\le \tilde{K}\}$\footnote{Here $\mbox{dist}(\bmu_t,B_{\tilde{K}})= \inf_{\boldsymbol y\in B_{\tilde{K}}} \|\bmu_t-\boldsymbol y\|$ is the distance between $\bmu_t$ and the ball $B_{\tilde{K}}$.}. As $\tilde{K}˜$ is independent on $t$ and the sequence of the cost
functions, we conclude existence of $M$ such that $\|\bmu_t\|$ is bounded almost surely by $M$, uniformly with respect to time $t$ and the sequences of cost functions $\{c_1, c_2, \ldots\}$. }
\end{proof}

With this lemma in place, we can prove the main result.

\begin{proof}[Proof of Theorem\r\ref{th:main}]
The plan of the proof is as follows.    {Let us denote the regret calculated with respect to the smoothed cost functions $\tilde c_t(\bmu)$  along the sequence $\bmu_t$ by $\tilde R^{\bmu_t}_{\bx}(T)$, namely, $\tilde R^{\bmu_t}_{\bx}(T) = \sum_{t=1}^T \tilde c_t (\bmu_t ) - \sum_{t=1}^T \tilde c_t (\bx)$}. We follow the idea of the proof in \cite{zink} and upperbound  $\tilde R^{\bmu_t}_{\bx}(T)$ by
 $\sum_{t=1}^T(\nabla \tilde c_t(\bmu_t),\bmu_t-\bx)$,
 where $\bx$ is any point from $\R^n$ such that $\|\bx\|$ is bounded. Then, using Lemma~\ref{lem:mixedstr_cost}, we relate back to the regret of the original cost functions evaluated at the mean vector $\bmu_t$,   {$R^{\bmu_t}_{\bx}(T)= \sum_{t=1}^T  c_t(\bmu_t) - \sum_{t=1}^T c_t(\bx)$}. Finally, taking conditional expectation of $R^{\bmu_t}_{\bx}$ and using Lemma~\ref{lem:mixedstr_cost}  we  get the bound on expected regret for the queried points, that is, $R_{\bx}(T) = \sum_{t=1}^T c_t (\xb_t ) - \sum_{t=1}^T c_t (\bx )$ as defined per \eqref{eq:regret}.

By  convexity of $\tilde c_t$ as shown in Lemma~\ref{lem:mixedstr_cost}, for any $\{\bmu_t\}, \bx \in \R^n$
 \begin{align}\label{eq:reg1}
  \sum_{t=1}^T &\tilde c_t(\bmu_t) - \sum_{t=1}^T \tilde c_t(\bx)\leq \sum_{t=1}^T(\nabla \tilde c_t(\bmu_t),\bmu_t-\bx).
   \end{align}
Thus, the regret calculated for the functions $\tilde c_t(\bmu)$ is at least as much as the regret calculated for the function $(\nabla \tilde c_t(\bmu_t), \bmu)$. To bound this term, analogously to equality \eqref{eq:firstest}, but evaluating $\| \bmu_{t+1} -\bx\| $ instead of $\| \bmu_{t+1}\|$  we can write
  {\begin{align}
\nonumber
&(\nabla \tilde c_t(\bmu_t), \bmu_t-\bx)=\frac{1}{2\alpha_t}(\|\bmu_t-\bx\|^2-\|\bmu_{t+1}-\bx\|^2)\\
 \label{eq:firstest1}
&\quad+ \frac{\alpha_t}{2}(\|\nabla \tilde c_t(\bmu_t)\|^2 + \|\xi_t(\xb_t,\bmu_t, \sigma_t)\|^2) \\
\nonumber
&\quad-(\xi_t(\xb_t,\bmu_t, \sigma_t),\bmu_t -\bx)-\alpha_t(\nabla \tilde  c_t(\bmu_t), \xi_t(\xb_t,\bmu_t, \sigma_t)).
\end{align}
By taking conditional expectation $\E_{t}$ of the both sides in the  equality \eqref{eq:firstest1}, we get that almost surely
\begin{align}
\nonumber
&(\nabla \tilde c_t(\bmu_t), \bmu_t-\bx)=\frac{1}{2\alpha_t}(\|\bmu_t-\bx\|^2-\E_{t}\{\|\bmu_{t+1}-\bx\|^2\})\\
 \label{eq:firstest1_1}
&\quad+ \frac{\alpha_t}{2}(\|\nabla \tilde c_t(\bmu_t)\|^2 + \E_{t}\{\|\xi_t(\xb_t,\bmu_t, \sigma_t)\|^2\}),
\end{align}
where we used the fact that $\E_{t}\xi_t(\xb_t,\bmu_t, \sigma_t)=0$ for all $t$, which is implied by \eqref{eq:martdiff}, and the following  property of the conditional expectation: $\E_{t}\{f(\bmu_{t})\}=\E\{f(\bmu_{t})|\EuScript F_{t}\}=f(\bmu_{t})$ almost surely for any $t$ and any continuous function $f$.
By summing up the equality \eqref{eq:firstest1_1} over $t=1, \dots, T$ we conclude that almost surely
\begin{align*}
\sum_{t=1}^{T}(\nabla \tilde c_t(\bmu_t), \bmu_t-\bx)=\sum_{t=1}^{T}&\frac{1}{2\alpha_t}(\|\bmu_t-\bx\|^2-\E_{t}\{\|\bmu_{t+1}-\bx\|^2\})\\
\nonumber
&+\sum_{t=1}^{T}\frac{\alpha_t}{2}(\E_{t}\{\|\xi_t(\xb_t,\bmu_t, \sigma_t)\|^2\} + \|\nabla \tilde c_t(\bmu_t)\|^2).
\end{align*}
}

  {Then taking the full expectation of  both sides, we obtain that $\forall T>0$,
\begin{align}
\nonumber
\E\sum_{t=1}^{T}(\nabla \tilde c_t(\bmu_t), \bmu_t-\bx)=\sum_{t=1}^{T}&\frac{1}{2\alpha_t}(\E\|\bmu_t-\bx\|^2-\E\|\bmu_{t+1}-\bx\|^2)\\
\label{eq:condexp}
&+\E\left(\sum_{t=1}^{T}\frac{\alpha_t}{2}(\E_{t}\{\|\xi_t(\xb_t,\bmu_t, \sigma_t)\|^2\} + \|\nabla \tilde c_t(\bmu_t)\|^2)\right).
\end{align}
Above, we used the tower rule for the conditional expectation, namely $$\E\{\E_{t}\{\|\bmu_{t+1}-\bx\|^2\}\}=\E\{\E\{\|\bmu_{t+1}-\bx\|^2|\EuScript F_t\}\} = \E\{ \|\bmu_{t+1}-\bx\|^2\}\;\quad \forall t.$$}   {
Furthermore, according to Lemma\r\ref{lem:lem1} and Assumption\r\ref{assum:inftybeh1} (Remark\r\ref{rem:Assum_inf}), there exists $M$ such that $\|\bmu_t\|\le M$ almost surely for all $t$. By taking into account Lemmas\r\ref{lem:mixedstr_cost},\r\ref{lem:234_moments} and inequality\r\eqref{eq:dim_depend}, we conclude that almost surely the term $\|\nabla \tilde c_t(\bmu_t)\|$ is bounded and $\E_{t}\{\|\xi_t(\xb_t,\bmu_t, \sigma_t)\|^2\}= O\left(\frac{nl^2\|\bmu_t\|^2}{\sigma_t^2} + 1\right)$. Hence,
\begin{align}
\nonumber
 &\E\sum_{t=1}^{T}(\nabla \tilde c_t(\bmu_t), \bmu_t-\bx)\le \frac{1}{2\alpha_1}\E\|\bmu_1-\bx\|^2 +\frac{1}{2}\sum_{t=2}^{T}\left(\frac{1}{\alpha_t}-\frac{1}{\alpha_{t-1}}\right)\E\|\bmu_t-\bx\|^2\\
\nonumber
&+\E\left(\sum_{t=1}^{T}\frac{\alpha_t}{2}(\E_{t}\{\|\xi_t(\xb_t,\bmu_t, \sigma_t)\|^2\} + \|\nabla \tilde c_t(\bmu_t)\|^2)\right)\\
\label{eq:condexp1}
&  {= 2(M^2+\|\bx\|^2)\frac{1}{2\alpha_T}  + nl^2 O\left(\sum_{t=1}^T\frac{\alpha_t}{\sigma_t^2}\right),}
\end{align}
where in the last inequality we used $\|\bmu_t-\bx\|^2\le 2(\|\bmu_t\|^2 + \|\bx\|^2)\le 2(M^2+\|\bx\|^2)$. Next, taking into account the settings for $\alpha_t$ and $\sigma_t$, we get
\begin{align*}
\sum_{t=1}^T\frac{\alpha_t}{\sigma_t^2}=\sum_{t=1}^T\frac{1}{(nt)^{a-2b}}& \le\frac{1}{n^{a-2b}}\left( 1 + \int_{1}^T\frac{dt}{t^{a-2b}}\right)\cr
&=\frac{1}{n^{a-2b}}\left(\frac{T^{1-a+2b}}{1-a+2b}-\frac{a-2b}{1-a+2b}\right),
\end{align*}
Thus,
\begin{align*}
 \E\sum_{t=1}^{T}(\nabla \tilde c_t(\bmu_t), \bmu_t-\bx)\le&(M^2+\|\bx\|^2)O\left({n^aT^a}\right) + l^2O\left((nT)^{1-a+2b}+n^{1-a+2b}c\right),
\end{align*}
where $c=-\frac{a-2b}{1-a+2b}$.
Hence, from \eqref{eq:reg1},
\begin{align}
\label{eq:reg_rate}
\E\left\{\frac{\tilde R^{\bmu_t}_{\bx}(T)}{T}\right\}= (M^2+\|\bx\|^2)O\left(\frac{n^a}{T^{1-a}}\right) + l^2O\left(\frac{n^{1-a+2b}}{T^{a-2b}}\right).
\end{align}
Next, according to Lemma\r\ref{lem:mixedstr_cost}, Part 2,
\begin{align}
\label{eq:reg_connect}
R^{\bmu_t}_{\bx}(T)& = \sum_{t=1}^T  c_t(\bmu_t) - \sum_{t=1}^T c_t(\bx)\\
\nonumber
& \le \sum_{t=1}^T(\tilde c_t(\bmu_t) + \frac{nL\sigma^2_t}{2})- \sum_{t=1}^T \tilde c_t(\bx)\\
\nonumber
& = \tilde R^{\bmu_t}_{\bx}(T) + \sum_{t=1}^T \frac{nL\sigma^2_t}{2}.
\end{align}
Hence, accounting for the inequality $\sum_{t=1}^T\sigma^2_t \le \frac{1}{n^{2b}}\left(\frac{T^{1-2b}}{1-2b}-\frac{2b}{1-2b}\right)$, we conclude that
\begin{align*}
\E\left\{\frac{R^{\bmu_t}_{\bx}(T)}{T}\right\}= (M^2+\|\bx\|^2)O\left(\frac{n^a}{T^{1-a}}\right) + l^2O\left(\frac{n^{1-a+2b}}{T^{a-2b}}\right)+ LO\left(\frac{n^{1-2b}}{T^{2b}}\right).
\end{align*}
As $0<a<1$, $b>0$, $a-2b>0$, the inequality above implies  $\varlimsup_{T\to\infty}\E \left\{\frac{R^{\bmu_t}_{\bx}(T)}{T}\right\}\le 0$ almost surely.
Notice that $R_{\bx}(T) = R^{\bmu_t}_{\bx}(T) +\sum_{t=1}^T (c_t (\xb_t ) - c_t (\bmu_t ))$.
Then, by taking expectation conditioned on $\EuScript F_T$, we get that almost surely
\begin{align*}
\E\{R_{\bx}(T)|\EuScript F_T\} &= R^{\bmu_t}_{\bx}(T) +\sum_{t=1}^T  \E\{(c_t (\xb_t ) - c_t (\bmu_t ))|\EuScript F_T\}\cr
& = R^{\bmu_t}_{\bx}(T) +\sum_{t=1}^T  (\tilde c_t (\bmu_t ) - c_t (\bmu_t ))\cr
&\le R^{\bmu_t}_{\bx}(T) + \frac{nL}{2} \sum_{t=1}^T  \sigma^2_t.
\end{align*}
In the inequality above we used Part 2 of Lemma\r\ref{lem:mixedstr_cost}.
Now by taking the full expectation and using the inequality $\sum_{t=1}^T\sigma^2_t \le \frac{1}{n^{2b}}\left(\frac{T^{1-2b}}{1-2b}-\frac{2b}{1-2b}\right)$, we conclude that
\begin{align}\label{eq:reg_rate_x}
  \E\left\{\frac{R_{\bx}(T)}{T}\right\} =(M^2+\|\bx\|^2)O\left(\frac{n^a}{T^{1-a}}\right) + l^2O\left(\frac{n^{1-a+2b}}{T^{a-2b}}\right)+ LO\left(\frac{n^{1-2b}}{T^{2b}}\right)
\end{align}
and, since the inequalities $2a-2b>1$ and $a<1$ imply $a-2b>0$, we get $\varlimsup_{T\to\infty} \E \left\{\frac{R_{\bx}(T)}{T}\right\}\le 0$ as desired. The result $$\E\left\{\frac{R_{\bx}(T)}{T}\right\} = (M^2+\|\bx\|^2+L+l^2)O\left(\frac{n^{\frac{2}{3}}}{T^{\frac{1}{3}}}\right)$$
follows from optimizing the rate in \eqref{eq:reg_rate_x} with respect to $a, b$ subject to the constraints $ 0 < a < 1$, $ 2a -2b > 1$, $0 < b < 1$.
}
\end{proof}

\begin{rem}\label{rem:expect}
Note that Lemma~\ref{lem:lem1} provides the result on the almost surely bounded iterates. However, in the proof of Theorem~\ref{th:main} we needed the weaker condition of the bounded expectation of $\|\bmu_t\|^2$ (see \eqref{eq:condexp1}).
\end{rem}

Let us further provide insights on the assumptions and their use in the proof above. Assumption\r\ref{assum:Lipsch}  on the uniform bound of gradients is used to show that through scaling the measured payoffs (zeroth order feedback) we obtain an estimate of gradient of the smoothed version of the cost in Lemma\r\ref{lem:zeta_martingale}. This assumption is also used to ensure the variance and higher order moments of $\xi_t$, perturbation of gradients in equation \eqref{eq:alg1}, are bounded (Lemma\r\ref{lem:234_moments}), and consequently to prove boundedness of the iterates (Lemma\r\ref{lem:lem1}). Assumption\r\ref{assum:Lipsch_grad} on uniformly Lipschitz gradients is used to bound the difference between the smoothed and the original cost functions through Lemma\r\ref{lem:mixedstr_cost} Part\r2. This enables us to bound the regret by first computing it along the smoothed version of the cost in the proof of Theorem\r\ref{th:main}.
  {Assumption\r\ref{assum:inftybeh1}  implies existence of a ball with  center at the origin and  radius $\tilde K$ whose complement contains no minima of the cost functions. The parameter $\tilde K$ defines, in its turn, a finite upper bound for the procedure's iterates (see the proof of Lemma~\ref{lem:lem1}). Thus, this assumption rules out possibility of infinite regret.}

{
\section{Two-point Feedback}\label{sec:twopoint}
Let us assume that given the process \eqref{eq:alg}, we can obtain the value of the function $c_t$ not only in the current state $\xb_t$, but also at the current mean value $\bmu_t$. The regret as defined in \eqref{eq:regret} is still in terms of the incurred costs $c_t(\xb_t)$ in comparison to any static choice $\xb$ with bounded norm.
In this section we modify the procedure \eqref{eq:alg}  in such a way that the upper bound for the regret achieves its optimum over the time and dimension parameters, namely is $O(\sqrt n\sqrt{T})$.
\subsection{Unconstrained Optimization}\label{subsec:two_unconstr}
We modify the process \eqref{eq:alg} as follows: We start with an arbitrary $\bmu_0$. Then, for $t=0,1,2,\ldots$
\begin{align}\label{eq:alg2P}
\xb_t &\sim\EuScript N(\bmu_t,\sigma_t),\\
\nonumber
 \bmu_{t+1} &= \bmu_t - \alpha_t (\hat c_t - c_t(\bmu_t)) \frac{\xb_t - \bmu_t}{\sigma^2_t},
\end{align}
where, as before, $\hat c_t = c_t(\xb_t)$.
The procedure above can be rewritten as
\begin{align}
\label{eq:alg1_2P}
 \bmu_{t+1} &= \bmu_t - \alpha_t\nabla\tilde c_t(\bmu_t)+\alpha_t \zeta_t(\xb_t,\bmu_t, \sigma_t),\\
\zeta_t(\xb_t,\bmu_t, \sigma_t) &= \nabla\tilde c_t(\bmu_t)-(\hat c_t - c_t(\bmu_t))\frac{\xb_t - \bmu_t}{\sigma^2_t}.
\end{align}
Analogously to Lemma\r\ref{lem:zeta_martingale}, we can formulate the following result.
\begin{lem}\label{lem:zeta_martingale2P}
Under Assumption\r\ref{assum:Lipsch},
\begin{align}\label{eq:martdiff2P}
\E_{t}\zeta_t(\xb_t,\bmu_t, \sigma_t) = \nabla\tilde c_t(\bmu_t)-\E_{t}\{\hat c_t\frac{\xb_t - \bmu_t}{\sigma^2_t}\} = 0.
\end{align}
\end{lem}
\begin{proof}
The first equality in \eqref{eq:martdiff2P} holds, due to the fact that
\begin{align*}
\E_{t}\zeta_t(\xb_t,\bmu_t, \sigma_t) &= \nabla\tilde c_t(\bmu_t)\ - E_{\xb_t}\{(\hat c_t-c_t(\bmu_t))\frac{\xb_t - \bmu_t}{\sigma^2_t}\},\\
\E_{t}\{c_t(\bmu_t)\frac{\xb_t - \bmu_t}{\sigma^2_t}\} &= 0.
\end{align*}
To get the second equality in \eqref{eq:martdiff2P}, we can repeat the proof of Lemma\r\ref{lem:zeta_martingale}.
\end{proof}

Further, we  can notice that Lemma\r\ref{lem:234_moments} can be reformulated in terms of the new stochastic term $\zeta_t(\xb_t,\bmu_t, \sigma_t)$ as follows.
  {\begin{lem}\label{lem:234_moments2P} Under Assumption~\ref{assum:Lipsch} the following estimations hold: 
 \begin{align*}
  \E_{t}\{\|\zeta_t(\xb_t,\bmu_t,\sigma_t)\|^2\}=O(l^2), \\
  \E_{t}\{\|\zeta_t(\xb_t,\bmu_t,\sigma_t)\|^3\}=O(l^3),\\
  \E_{t}\{\|\zeta_t(\xb_t,\bmu_t,\sigma_t)\|^4\} = O(l^4),
 \end{align*}
 Moreover,
  \begin{align}\label{eq:dim_depend2P}
  \E_{t}\{\|\zeta_t(\xb_t,\bmu_t,\sigma_t)\|^2\}= O(nl^2).
  \end{align}
\end{lem}
}

\begin{proof}
 See Appendix\r\ref{app:integral_proofs}.
\end{proof}

The fact that the unbiased estimation
 $$(\hat c_t - c_t(\bmu_t))\frac{\xb_t - \bmu_t}{\sigma^2_t} = (c_t(\xb_t) - c_t(\bmu_t))\frac{\xb_t - \bmu_t}{\sigma^2_t}$$
 of the gradient $\nabla\tilde c_t(\bmu_t)$ uses two points implies its bounded moments, whereas in the case of one-point feedback we can only upper bound the moments by some functions dependent on $\bmu_t$ (compare Lemma\r\ref{lem:234_moments2P} with Lemma\r\ref{lem:234_moments}).
 This feature of the approach based on two-point feedback allows us to relax the conditions on the parameters $\alpha_t$ and $\sigma_t$ in Lemma\r\ref{lem:lem1} to guarantee the bounded iterations in the new process in \eqref{eq:alg2P}.

\begin{lem}\label{lem:lem1_2P} Consider Procedure  \eqref{eq:alg2P} with step-size $\{\alpha_t = O\left(\frac{1}{t^a}\right)\}$ and variance $\{\sigma_t= O\left(\frac{1}{t^b}\right)\}$, where $0<a<1$, $b>0$. There exists a finite constant $M'$:
$$\Pr\{\|\bmu_t\|\le M', t=1,2,\ldots\,|\,\|\bmu_{0}\|<\infty\}=1$$
for any sequence $\{c_1,c_2,\ldots\}$, for which Assumptions\r\ref{assum:Lipsch}, \ref{assum:inftybeh1} hold.
\end{lem}
\begin{proof}
Analogously to the proof of Lemma\r\ref{lem:lem1}, we consider the function $V(\bmu) = W(\|\bmu\|^2)$, where $W:\R\to\R$ is defined as follows:
\begin{align*}
W(x) = \begin{cases}
        0, &\mbox{ if } x<\tilde K,\\
        (x-\tilde K)^2, &\mbox{ if } x\ge \tilde K,
       \end{cases}
\end{align*}
and $\tilde K$ is the constant from Lemma\r\ref{lem:mixedstr_cost}. We then continue with the exact same derivation as in Lemma\r\ref{lem:lem1} to show that $V(\bmu_t)$ is a nonnegative martingale by bounding the term $  \E\{V(\bmu_{t+1})|\bmu_t=\bmu\} - V(\bmu)$ as follows:
 \begin{align*}
 \EuScript LV (\bmu)\le &-2\alpha_t(\nabla \tilde c_t(\bmu),\bmu)W'(\|\bmu\|^2) + (g_1(t)+g_2(t))(1+V(\bmu)).
 \end{align*}
 The only difference between the above bound and the one in  \eqref{eq:V_onepoint} using  the one-point feedback  is that here $g_1(t) = O\left({\alpha_t^2}\right)$ and $g_2(t) = O\left({\alpha_t^2}\right)$ and hence, these terms do not exhibit the dependence on the variance parameter $\sigma_t$. This difference is due to bounds on $\zeta$ and $\xi$ using Lemmas\r\ref{lem:234_moments2P} and \r\ref{lem:234_moments}, respectively.
Due to $0 < a < 1$, we have that $\sum_{t=1}^{\infty} \alpha_t = \infty$, $\sum_{t=1}^{\infty} \alpha^2_t < \infty$. Hence, $V(\bmu_t)$ is a nonnegative martingale. Repeating the same reasoning analogous to the one in Lemma\r\ref{lem:lem1} we get the result.
\end{proof}

With this lemma in place, we can prove the main result for the process \eqref{eq:alg2P}.

\begin{theorem}\label{th:main2P}
Let \eqref{eq:alg2P} define the optimization algorithm for the unconstrained online convex optimization problem $(\R^n,\{c_1,c_2,\ldots\})$. Choose the step-sizes and variances according to   $\{\alpha_t = \frac{1}{(nt)^a}\}$, $\{\sigma_t= \frac{1}{(nt)^b}\}$, where $0<a<1$, $b>0$. Then,  under Assumptions\r\ref{assum:Lipsch}-\ref{assum:inftybeh1},
the regret of  algorithm \eqref{eq:alg} estimated with respect to the query points $\{\xb_t\}$, that is, $R_{\bx}(T)= \sum_{t=1}^T c_t (\xb_t ) - \sum_{t=1}^T c_t (\bx )$ satisfies
\begin{align*}
\E\left\{\frac{R_{\bx}(T)}{T}\right\}= ({M'}^2+\|\bx\|^2)O\left(\frac{n^a}{T^{1-a}}\right) + l^2O\left(\frac{n^{1-a}}{T^{a}}\right)+ LO\left(\frac{n^{1-2b}}{T^{2b}}\right).
\end{align*}
In particular,  $$ \varlimsup_{T\to\infty} \E \left\{\frac{R_{\bx}(T)}{T}\right\}\le 0,$$ and for the optimal choice of $a \rightarrow \frac{1}{2}$, $b \rightarrow \frac{1}{4}$, the  regret rate is
\begin{align*}
\E\left\{\frac{R_{\bx}(T)}{T}\right\} = ({M'}^2+\|\bx\|^2+L+l^2)O\left(\frac{\sqrt{n}}{\sqrt{T}}\right).
\end{align*}

\end{theorem}

\begin{proof}
  {Let us recall $\tilde R^{\bmu_t}_{\bx}(T) := \sum_{t=1}^T \tilde c_t (\bmu_t ) - \sum_{t=1}^T \tilde c_t (\bx)$ and $R^{\bmu_t}_{\bx}(T):= \sum_{t=1}^T  c_t(\bmu_t) - \sum_{t=1}^T c_t(\bx)$ from Proof of Theorem\r\ref{th:main}}. As per proof of Theorem\r\ref{th:main}, we bound the regret at the linearized costs based on the following observation:
\begin{align}\label{eq:reg1_2P}
\sum_{t=1}^T &\tilde c_t(\bmu_t) - \sum_{t=1}^T \tilde c_t(\bx)\leq \sum_{t=1}^T(\nabla \tilde c_t(\bmu_t),\bmu_t-\bx),
 \end{align}
  {Furthermore, repeating the steps of Theorem\r\ref{th:main}'s proof, we obtain that $\forall T>0$
\begin{align}
\label{eq:condexp_2P}
\E\sum_{t=1}^{T}&(\nabla \tilde c_t(\bmu_t), \bmu_t-\bx)=\sum_{t=1}^{T}\frac{1}{2\alpha_t}(\E\|\bmu_t-\bx\|^2-\E\|\bmu_{t+1}-\bx\|^2)+\\
\nonumber
+&\E\sum_{t=1}^{T}\frac{\alpha_t}{2}(\E_{t}\|\zeta_t(\xb_t,\bmu_t,\sigma_t)\|^2 + \|\nabla \tilde c_t(\bmu_t)\|^2).
\end{align}
According to Lemma\r\ref{lem:lem1_2P} and Assumption\r\ref{assum:inftybeh1} (Remark\r\ref{rem:Assum_inf}), there exists $M'$ such that $\|\bmu_t-\bx\|\le M'+\|\bx\|$ almost surely for all $t$. By taking into account Lemmas\r\ref{lem:mixedstr_cost},\r\ref{lem:234_moments2P} and inequality\r\eqref{eq:dim_depend2P}, we conclude that almost surely the term $\|\nabla \tilde c_t(\bmu_t)\|$ is bounded. Furthermore, $\E_{t}\|\zeta_t(\xb_t,\bmu_t,\sigma_t)\|^2= O(nl^2)$ -- note the difference between this bound and that of $\E_{t}\{\|\xi_t(\xb_t,\bmu_t, \sigma_t)\|^2\}= O\left(\frac{nl^2\|\bmu_t\|^2}{\sigma_t^2} + 1\right)$, in the one-point feedback setting.  Thus, repeating the derivation in \eqref{eq:condexp1} we have
\begin{align*}
&\E\sum_{t=1}^{T}(\nabla \tilde c_t(\bmu_t), \bmu_t-\bx)
 = 2({M'}^2+\|\bx\|^2)\frac{1}{2\alpha_T}  + nl^2 O\left(\sum_{t=1}^T{\alpha_t}\right).
\end{align*}
Note the difference between the above upper bound and that in \eqref{eq:condexp1} once again due to the tightened bound on the variance of the noise $\xi_t$ in the two-point feedback setting here.
Next, taking into account the setting for $\alpha_t$, we get
\begin{align*}
\sum_{t=1}^T{\alpha_t}=\sum_{t=1}^T\frac{1}{(nt)^{a}}& \le\frac{1}{n^a}\left( 1 + \int_{1}^T\frac{dt}{t^{a}}\right)=\frac{1}{n^a}\left(\frac{T^{1-a}}{1-a}-\frac{a}{1-a}\right),
\end{align*}
Thus,
\begin{align*}
\E\sum_{t=1}^{T}(\nabla \tilde c_t(\bmu_t), \bmu_t-\bx)=& ({M'}^2+\|\bx\|^2)O\left({n^aT^a}\right) + l^2O\left((nT)^{1-a}+n^{1-a}c\right),
\end{align*}
where $c=-\frac{a}{1-a}$.
From \eqref{eq:reg1_2P}, it follows that
\begin{align}
\label{eq:reg_rate2P}
\E\left\{\frac{\tilde R^{\bmu_t}_{\bx}(T)}{T}\right\}=  ({M'}^2+\|\bx\|^2)O\left(\frac{n^a}{T^{1-a}}\right) + l^2O\left(\frac{n^{1-a}}{T^{a}}\right).
\end{align}
Then, similar to \eqref{eq:reg_connect} we connect the regret between the smoothed cost $\tilde{c}$ and the original cost $c$ using Lemma\r\ref{lem:mixedstr_cost}, Part 2,
\begin{align}
\label{eq:reg_connect2P}
R^{\bmu_t}_{\bx}(T)& = \sum_{t=1}^T  c_t(\bmu_t) - \sum_{t=1}^T c_t(\bx) \le \tilde R^{\bmu_t}_{\bx}(T) + \sum_{t=1}^T \frac{nL\sigma^2_t}{2}.
\end{align}
Hence, accounting for the inequality $\sum_{t=1}^T\sigma^2_t \le \frac{1}{n^{2b}} \left(\frac{T^{1-2b}}{1-2b}-\frac{2b}{1-2b}\right)$,  we have
\begin{align*}
\E\left\{\frac{R^{\bmu_t}_{\bx}(T)}{T}\right\}= ({M'}^2+\|\bx\|^2)O\left(\frac{n^a}{T^{1-a}}\right) + l^2O\left(\frac{n^{1-a}}{T^{a}}\right)+ LO\left(\frac{n^{1-2b}}{T^{2b}}\right).
\end{align*}
As $0<a<1$ and $b>0$, the inequality above implies  $\varlimsup_{T\to\infty}\frac{R^{\bmu_t}_{\bx}(T)}{T}\le 0$ almost surely. Consequently,  from the reasoning analogous to one in proof of the corresponding parts in Theorem\r\ref{th:main} we have
$$ \varlimsup_{T\to\infty} \E \left\{\frac{R_{\bx}(T)}{T}\right\}\le 0,$$ and optimizing over the choices of $a, b$, we get the desired rate:
\begin{align*}
\E\left\{\frac{R_{\bx}(T)}{T}\right\} =({M'}^2+\|\bx\|^2+L+l^2) O\left(\frac{\sqrt{n}}{\sqrt{T}}\right).
\end{align*}}
\end{proof}

\section{Conclusion}\label{sec:conclusion}
We provided a novel algorithm for the bandit online optimization problem with convex cost functions over  unconstrained  action spaces. Our algorithm was based on a zero-order oracle. In the case of one query point, we achieved a regret rate of  {$O(n^{2/3}T^{2/3})$}. We showed how the algorithm can be adopted to address constrained action spaces, achieving the same regret rate above. Moreover, we presented a version of the algorithm adapted to the setting with two-point feedback. For this case, we showed that by appropriately choosing the two points and the step-size and variance parameters, the proposed algorithm achieves the theoretical lower bound with respect to the number of queries, namely {$O(n^{1/2}T^{1/2})$}.

 \appendix
 \section{Supporting Theorems}\label{app:support}
To prove convergence of the algorithm we will use the results on convergence properties of the Robbins-Monro  stochastic approximation procedure analyzed in \cite{NH}.

We start by introducing some important notation.
Let $\{\zbx(t)\}_t$, $t\in \Z_+$, be a discrete-time Markov process on some state space $E\subseteq \R^n$, namely $\zbx(t)=\zbx(t,\omega):\Z_+\times\Omega\to E$, where $\Omega$ is the sample space of the probability space on which the process $\zbx(t)$ is defined. The transition function of this chain, namely $\Pr\{\zbx(t+1)\in\Gamma| \zbx(t)=\zbx\}$, is denoted by $P(t,\zbx,t+1,\Gamma)$, $\Gamma\subseteq E$.

\begin{defin}\label{def:def1}
The operator $L$ defined on the set of measurable functions $V:\Z_+\times E\to \R$, $\zbx\in E$, by
\begin{align*}
LV(t,\zbx)&=\int{P(t,\zbx,t+1,dy)[V(t+1,y)-V(t,\zbx)]}\cr
&=E[V(t+1,\zbx(t+1))\mid \zbx(t)=\zbx]-V(t,\zbx),
\end{align*}
is called a \emph{generating operator} of a Markov process $\{\zbx(t)\}_t$.
\end{defin}

  {
Now, we recall the following theorem for discrete-time Markov processes, which is proven in [13], Theorem 2.7.1. This theorem requires the following notations:
Let for any set $\EuScript B \subseteq E$ the set $U_{\epsilon,R}(\EuScript B)$ be $U_{\epsilon,R}(\EuScript B) = U_{\epsilon}(\EuScript B) \cap \{\zbx: \zbx < R\}$, where
$U_{\epsilon}(\EuScript B) = {\zbx : \mbox{dist}(\zbx,\EuScript B) < \epsilon}$.

\begin{theorem}\label{th:finiteness}
  Consider a Markov process $\{\zbx(t)\}_t$ and suppose that there exists a function $V(t,\zbx)\ge 0$ such that $\inf_{t\ge0}V(t,\zbx)\to\infty$ as $\|\zbx\|\to\infty$ and
  \[LV(t,\zbx)\le -\alpha(t+1)\psi(t,\zbx) + f(t)(1+V(t,\zbx)),\]
   where $\psi\ge 0$ on $\R \times \R^n$ and $\psi > 0$ on $\R \times \R^n \setminus \EuScript B$ for some set $\EuScript B \subset \R^n$. Further assume that the functions $f(t)$ and $\alpha(t)$ satisfy $f(t)>0$, $\sum_{t=0}^{\infty}f(t)<\infty$, and $\alpha(t)>0$, $\sum_{t=0}^{\infty} \alpha(t)= \infty$.
   Let $\inf_{\zbx\in U_{\epsilon,R}(\EuScript B)} V (\zbx) > 0$ for any $R > \epsilon > 0$ and $V (\zbx) = 0$ for $\zbx \in \EuScript B$. Moreover, let
   $\lim_{\zbx\to\EuScript B} V (\zbx) = 0$. Then the Markov process $\{\zbx(t)\}_t$ converges to the set $\EuScript B$ almost  surely as $t \to \infty$.
\end{theorem}
}
%

\section{Proofs of Lemmas}\label{app:integral_proofs}

\begin{proof}(\emph{of Lemma\r\ref{lem:zeta_martingale}})
First, we show that under Assumption\r\ref{assum:Lipsch} we can differentiate $\tilde c_t(\bmu_t)$ defined by \eqref{eq:mixedstr_cost} with respect to the parameter $\bmu_t$ under the integral sign. Note that this was stated as a fact in \cite{NesterovGrFree} without a proof. Here, we provide a proof for completeness.
Indeed, let us consider the integral, which we obtain, if we formally differentiate the function under the integral \eqref{eq:mixedstr_cost} with respect to $\bmu_t$, namely
\begin{align}\label{eq:integral}
 \frac{1}{\sigma_t^2}\int_{\R^n}c_t(\bx)(\bx - \bmu_t)p(\bmu_t,\sigma_t,\bx)d\bx.
\end{align}
The function under the integral sign, $c_t(\bx)(\bx - \bmu_t)p(\bmu_t,\sigma_t,\bx)$, is continuous given Assumption\r\ref{assum:Lipsch}. Thus, it remains to check that the integral of this function converges uniformly with respect to $\bmu_t$ over the whole $\R^n$. We can write the Taylor expansion of the function $c_t$ around the point $\bmu_t$ in the integral \eqref{eq:integral}:
 \begin{align*}
  &\int_{\R^n}c_t(\bx)(\bx - \bmu_t)p(\bmu_t,\sigma_t,\bx)d\bx \cr
  &= \int_{\R^n}(c_t(\bmu_t) + (\nabla c_t(\boldsymbol{\eta}(\bx,\bmu_t)),\bx - \bmu_t))(\bx-\bmu_t)p(\bmu_t,\sigma_t,\bx)d\bx\cr
  &=\int_{\R^n}(\nabla c_t(\boldsymbol{\eta}(\bx,\bmu_t)),\bx - \bmu_t)(\bx-\bmu_t)p(\bmu_t,\sigma_t,\bx)d\bx\cr
  &=\int_{\R^n}(\nabla c_t(\tilde{\boldsymbol{\eta}}(\by,\bmu_t)),\by)\by p(\boldsymbol 0,\sigma_t,\by)d\by,
 \end{align*}
where $\boldsymbol{\eta}(\bx,\bmu_t)=\bmu_t+\theta(\bx-\bmu_t)$, $\theta\in(0,1)$, $\by = \bx-\bmu_t$, $\tilde{\boldsymbol{\eta}}(\by,\bmu_t) = \bmu_t +\theta\by$.
The uniform convergence of the integral above follows from the fact\footnote{see the basic sufficient condition using majorant \cite{zorich}, Chapter 17.2.3.} that, under Assumption\r\ref{assum:Lipsch}, $\nabla c_t(\tilde{\boldsymbol{\eta}}(\by,\bmu_t))\le l$ for some positive constant $l$ and, hence,
\[|(\nabla\phi(\tilde{\boldsymbol{\eta}}(\by,\bmu_t)),\by)\by p(\boldsymbol 0,\sigma_t,\by)|\le h(\by)=l\|\by\|^2 p(\boldsymbol 0,\sigma_t,\by),\]
where $\int_{\R^N}h(\by)d\by<\infty$.
Thus, part 1 of the Lemma follows from
\begin{align*}
\nabla\tilde c_t&(\bmu_t) = \int_{\R^n} \nabla_{\bmu_t} \big(c_t(\bx) p(\bmu_t,\sigma_t,\bx)\big)d\bx\cr
&= \int_{\R^n}c_t(\bx)\frac{\bx - \bmu_t}{\sigma^2_t}p(\bmu_t,\sigma_t,\bx)d\bx
\end{align*}
  {and the fact that
$$\E_t\{\hat c_t\frac{\xb_t- \bmu_t}{\sigma^2_t}\} = \E\{c_t(\xb_t)\frac{\xb_t - \bmu_t}{\sigma^2_t}|\xb_t\sim \EuScript N(\bmu_t,\sigma_t)\}.$$}
Furthermore, for each $k$th coordinate of the above vector $c_t(\bx)\frac{\bx - \bmu_t}{\sigma^2_t}p(\bmu_t,\sigma_t,\bx)$, $k\in[n]$, and given $x^{-k}=(x^1,\ldots,x^{k-1},x^{k+1},\ldots,x^n)$ we get
\begin{align}
\label{eq:mixedstr_grad1}
&\int_{\R^n}c_t(\bx)\frac{x^k - \mu^k_t}{\sigma^2_t}p(\bmu_t,\sigma_t,\bx)d\bx \cr
&=\frac{-1}{(\sqrt{2\pi}\sigma_t)^{n}}\int_{\R^{n-1}}\left[\int_{x^k=-\infty}^{x^k=+\infty}c_t(\bx)d\left(\exp\left\{-\sum_{k=1}^{n}\frac{(x^k-\mu^k_t)^2}{2\sigma_t^2}\right\}\right)\right]\cr
&\qquad\qquad\qquad\times\exp\left\{-\sum_{j\ne k}^{n}\frac{(x^j-\mu^j_t)^2}{2\sigma_t^2}\right\} d x^{-k}=\int_{\R^n}\frac{\partial c_t(\bx)}{\partial x^k}p(\bmu_t,\sigma_t,\bx)d\bx,
\end{align}
where in the above, we use integration by parts and the fact that $c_t$ grows at most linearly as $\| \xb \| \rightarrow \infty$ to get to the last equality.
Thus, the claim of the lemma  follows.
\end{proof}

\begin{proof}(\emph{of Lemma\r\ref{lem:mixedstr_cost}})
Part 1.
\begin{align*}
&\tilde c_t(a\bmu_1 + (1-a)\bmu_2) = \frac{1}{(2\pi\sigma^2)^{n/2}}\cr
&\times\int_{\R^n}c_t(\bx) \exp\left\{-\frac{\|\bx-a\bmu_1 - (1-a)\bmu_2\|^2}{2\sigma^2}\right\}d\bx.
\end{align*}
By the substitution $\by = \bx-a\bmu_1 - (1-a)\bmu_2$ we get
\begin{align*}
&\int_{\R^n}c_t(\bx) \exp\left\{-\frac{\|\bx-a\bmu_1 - (1-a)\bmu_2\|^2}{2\sigma^2}\right\}d\bx\cr
& = \int_{\R^n}c_t(\by + a\bmu_1 +(1-a)\bmu_2) \exp\left\{-\frac{\|\by\|^2}{2\sigma^2}\right\}d\by\cr
& = \int_{\R^n}c_t(a(\by + \bmu_1) +(1-a)(\by + \bmu_2)) \exp\left\{-\frac{\|\by\|^2}{2\sigma^2}\right\}d\by\cr
& \le a\int_{\R^n}c_t(\by + \bmu_1)\exp\left\{-\frac{\|\by\|^2}{2\sigma^2}\right\}d\by\cr &\quad +(1-a)\int_{\R^n}c_t(\by + \bmu_2) \exp\left\{-\frac{\|\by\|^2}{2\sigma^2}\right\}d\by\cr
& = a\int_{\R^n}c_t(\bx)\exp\left\{-\frac{\|\bx - \bmu_1\|^2}{2\sigma^2}\right\}d\bx\cr
& \quad +(1-a)\int_{\R^n}c_t(\bx) \exp\left\{-\frac{\|\bx- \bmu_2\|^2}{2\sigma^2}\right\}d\bx\cr
& = a\tilde c_t(\bmu_1) + (1-a)\tilde c_t(\bmu_2).
\end{align*}
Hence,
\begin{align*}
\tilde c_t(a\bmu_1 + (1-a)\bmu_2)  \leq a\tilde c_t(\bmu_1) + (1-a)\tilde c_t(\bmu_2).
\end{align*}
The fact that for any fixed $t$ the gradient $\nabla\tilde c_t(\bmu_t)$ is bounded follows directly from the equation \eqref{eq:martdiff1} and Assumption\r\ref{assum:Lipsch} implying bounded $\nabla c_t$ on $\R^n$.

Part 2. By using the Taylor series expansion for the function $c_t$ around the vector $\bmu_t$ and from Assumption 2, we obtain
\begin{align*}
c_t(\bx) \le c_t(\bmu_t) + (\nabla c_t(\bmu_t),\bx - \bmu_t) + L/2 \|\bx-\bmu_t\|^2.
\end{align*}
Hence,
\begin{align*}
\; &|c_t(\bmu_t) - \tilde c_t(\bmu_t)| = |\int_{\R^{n}}[c_t(\bx)-c_t(\bmu_t)]p(\bmu_t, \sigma_t, \bx)d\bx |\cr
&\le |\int_{\R^{n}}(\nabla c_t(\bmu_t),\bx - \bmu_t)p(\bmu_t, \sigma_t, \bx)d\bx \cr
& +\int_{\R^{n}}L/2 \|\bx-\bmu_t\|^2 p(\bmu_t, \sigma_t, \bx) d\bx|= \frac{nL\sigma_t^2}{2}.
\end{align*}
as desired. Note that \cite{NesterovGrFree} showed this result with a slightly different technique.

Part 3. The existence of a finite constant $\tilde K>0$ such that $(\bmu,\nabla \tilde c_t(\bmu))>0$ for $\|\bmu\|^2>\tilde K$ follows from the fact that due to the bounded variance $\sigma^2_t$ (and, hence, coercivity of $\tilde c_t$, due to Assumption\r\ref{assum:inftybeh1} (Remark\r\ref{rem:Assum_inf}) and Part 2 shown above)  the argumentation analogous to one in Remark\r\ref{rem:Assum_inf} holds for functions $\tilde c_t(\bmu)$ on $\R^n$.
\end{proof}

\begin{proof}(\emph{of Lemma\r\ref{lem:234_moments}})
First note that, according to the Lyapunov's inequality,
\begin{align}
  \label{eq:Lyap1}&(\E_{t}\{\|\xi_t(\xb_t,\bmu_t,\sigma_t)\|^3\})^{1/3}\le(\E_{t}\{\|\xi_t(\xb_t,\bmu_t,\sigma_t)\|^4\})^{1/4},\\
  \label{eq:Lyap2}&(\E_{t}\{\|\xi_t(\xb_t,\bmu_t,\sigma_t)\|^2\})^{1/2}\le(\E_{t}\{\|\xi_t(\xb_t,\bmu_t,\sigma_t)\|^4\})^{1/4}.
\end{align}
Thus, it suffices to demonstrate that
\begin{align}\label{eq:fourthmom}
  \E_{t}\{\|\xi_t(\xb_t,\bmu_t,\sigma_t)\|^4\}\le \frac{f_3(\bmu_t, \sigma_t)}{\sigma_t^4},
 \end{align}
 where $f_3(\bmu_t, \sigma_t)$ is a polynomial of $\sigma_t$ and a fourth order polynomial of $\mu^i_t$, $i\in[n]$.

Let us consider any random vector $\boldsymbol X=(X_1,\ldots, X_n)\in\R^n$. The fourth central moment of this vector can be bounded as follows:
\begin{align}\label{eq:eq1}
  \E\|\boldsymbol X&-\E \boldsymbol X\|^4 = \E(\sum_{i=1}^{n}(X_i-\E X_i)^2)^2\cr
  = &\sum_{i=1}^{n}\E(X_i - \E X_i)^4
  + 2\sum_{i,j:i<j}\E\{(X_i-\E X_i)^2(X_j-\E X_j)^2\}\cr
  \le &\sum_{i=1}^{n}\E(X_i - \E X_i)^4
  + 2\sum_{i,j:i<j}\sqrt{\E(X_i-\E X_i)^4}\sqrt{\E(X_j-\E X_j)^4},
\end{align}
where in the last inequality we used the Hoelder inequality.
Hence, we proceed with estimating $\E(X_i - \E X_i)^4$. By opening brackets we get
\begin{align}\label{eq:eq2}
  \E(X_i - \E X_i)^4 & = \E X_i^4 + 6 (\E X_i)^2 \E X_i^2 - 4\E X_i \E X_i^3 - 3 (\E X_i)^4\cr
  & \le \E X_i^4 + 6 (\E X_i)^2 \E X_i^2- 4\E X_i \E X_i^3 \le 11\E X_i^4,
\end{align}
where in the last inequality we used the fact that $(\E X_i)^2\le\E X_i^2$ and the Lyapunov's inequalities (namely \eqref{eq:Lyap1}, \eqref{eq:Lyap2} with $\xi_t$ replaced by $X_i$) to get
\begin{align*}
-\E X_i\E X_i^3 \le \E|X_i| \E|X_i|^3\le\E X_i^4.
\end{align*}

According to \eqref{eq:martdiff},
\begin{align}\label{eq:xi}
\xi_t(\xb_t,\bmu_t,\sigma_t) = \E\left\{\hat c_t\frac{\xb_t - \bmu_t}{\sigma^2_t}\right\}-\hat c_t\frac{\xb_t - \bmu_t}{\sigma^2_t}.
\end{align}
Thus, according to \eqref{eq:eq1} and \eqref{eq:eq2},  to get \eqref{eq:fourthmom}, it remains to bound the fourth moment of $\eta_i=\eta_i(\xb_t,\bmu_t,\sigma_t) = \hat c_t\frac{x^i_t - \mu_t^i}{\sigma^2_t}$ for all $i=1,\ldots,n$, given that $\xb_t$ has the normal distribution with the parameters $\bmu_t$, $\sigma_t$. According to Assumption\r\ref{assum:Lipsch} and the properties of the central moments of the normal distribution\footnote{Here we use the fact that for any $i,j,k,l=1,\ldots,m$ and  $s_1,s_2,s_3,s_4 \in\{0,1,2,3,4\}$ such that $s_1+s_2+s_3+s_4=4$, the following relation holds: $\int_{\mathbb R^{n}}(x^i)^{s_1}(x^j)^{s_2}(x^k)^{s_3}(x^l)^{s_4}{(x^i - \mu^i_t)^4}p(\bmu_t,\sigma_t,\bx)d\bx = \sigma^4\tilde p(\sigma_t, \bmu_t)$, where $\tilde p(\sigma_t, \bmu_t)$ is a polynomial of $\sigma_t$ and not higher than a fourth order polynomial of $\mu^i_t$, $\mu^j_t$, $\mu^k_t$, and $\mu^l_t$.}, for any $i$ there exists a finite constant $K$ and  a function $\tilde f_i(\bmu_t, \sigma_t)$, which is a
polynomial of $\sigma_t$ and  a fourth order polynomial of $\bmu_t$, such that the following holds:
\begin{align*}
  \E_{t}\{\eta_i^4\}\le & K \int_{\mathbb R^{n}}(\sum_{i=1}^{n}x^i)^4\frac{(x^i - \mu^i_t)^4}{\sigma_t^8} p(\bmu_t,\sigma_t,\bx)d\bx
  \le  \frac{\tilde f_i(\bmu_t, \sigma_t)}{\sigma_t^4}.
\end{align*}
This inequality together with the inequalities \eqref{eq:eq1}, \eqref{eq:eq2}, and \eqref{eq:xi} imply \eqref{eq:fourthmom}.

Finally, we notice that due to Lemma\r\ref{lem:zeta_martingale}
\[\E_{t}\|\xi_t(\xb_t,\bmu_t,\sigma_t)\|^2\le \E_{t}\left\{c_t^2(\xb_t)\frac{\|\xb_t-\bmu_t\|^2}{\sigma^4_t}\right\}.\]
By taking Assumption\r\ref{assum:Lipsch} into account, we estimate further the right hand side of the inequality above as
\begin{align*}
  &\E_{t}\left\{c_t^2(\xb_t)\frac{\|\xb_t-\bmu_t\|^2}{\sigma^4_t}\right\}\cr
  &\le l^2 \int_{\R^n}\|\bx\|^2\frac{\|\bx-\bmu_t\|^2}{\sigma^4_t}p(\bmu_t,\sigma_t,\bx)d\bx + m_2,
  \end{align*}
where $m_1$ are some positive constant. Note that for any $i,j=1,\ldots,n$, $i\ne j$,
\[\int_{\R^n} x_i^2(x_i-\mu_t^i)^2 p(\bmu_t,\sigma_t,\bx) d\bx = \sigma_t^4 + {(\mu_t^i)}^2\sigma_t^2,\]
\[\int_{\R^n} x_i^2(x_j-\mu_t^j)^2 p(\bmu_t,\sigma_t,\bx) d\bx = \sigma_t^4 + {(\mu_t^i)}^2\sigma_t^2.\]
Thus, we can conclude that
\[\E_{t}\|\xi_t(\xb_t,\bmu_t,\sigma_t)\|^2= O\left(\frac{nl^2\|\bmu_t\|^2}{\sigma_t^2} + 1\right).\]
\end{proof}

\begin{proof}(\emph{of Lemma\r\ref{lem:234_moments2P}})
   Note that, analogously to \eqref{eq:eq2}, due to Lemma\r\ref{lem:zeta_martingale2P}, we get
   \begin{align*}
   \E_{t}\{\|\zeta_t(\xb_t,\bmu_t,\sigma_t)\|^4\} &\le 11 \E_{t}\{\|(\hat c_t - c_t(\bmu_t))\frac{\xb_t - \bmu_t}{\sigma^2_t}\|^4\} \cr
   &= 11\int_{\R^n}\frac{1}{\sigma^8}(c_t(\bx) - c_t(\bmu_t))^4\|\bx - \bmu_t\|^4 p(\bmu_t,\sigma_t,\bx)d\bx.
  \end{align*}
  Next, due to Assumption\r\ref{assum:Lipsch} and Remark\r\ref{rem:linbeh},
  \[(c_t(\bx) - c_t(\bmu_t))^4\le l^4\|\xb_t-\bmu_t\|^4.\]
  Thus,
  \[\E_{t}\{\|\zeta_t(\xb_t,\bmu_t,\sigma_t)\|^4\} \le 11 l^4\int_{\R^n}\frac{1}{\sigma^8}\|\xb_t - \bmu_t\|^8 p(\bmu_t,\sigma_t,\bx)d\bx = O(l^4) \]
  for some positive constant $C_3$.
  The rest follows from the arguments analogous in proof of Lemma\r\ref{lem:234_moments} above.
\end{proof}

%
%

%
%

\bibliographystyle{siamplain}

\bibliography{../../CentralPath/online_opt}

\begin{thebibliography}{10}

\bibitem{agarwal2010optimal}
{\sc A.~Agarwal, O.~Dekel, and L.~Xiao}, {\em Optimal algorithms for online
  convex optimization with multi-point bandit feedback.}, in Conference on
  Learning Theory (COLT), 2010, pp.~28--40.

\bibitem{BachPerchet}
{\sc F.~R. Bach and V.~Perchet}, {\em Highly-smooth zero-th order online
  optimization}, in Conference on Learning Theory (COLT), 2016, pp.~257--283.

\bibitem{bubeck2012regret}
{\sc S.~Bubeck, N.~Cesa-Bianchi, et~al.}, {\em Regret analysis of stochastic
  and nonstochastic multi-armed bandit problems}, Foundations and
  Trends{\textregistered} in Machine Learning, 5 (2012), pp.~1--122.

\bibitem{bubeck2017kernel}
{\sc S.~Bubeck, Y.~T. Lee, and R.~Eldan}, {\em Kernel-based methods for bandit
  convex optimization}, in Proceedings of the ACM SIGACT Symposium on Theory of
  Computing, 2017, pp.~72--85.

\bibitem{cutkosky2017online}
{\sc A.~Cutkosky and K.~Boahen}, {\em Online learning without prior
  information}, in Conference on Learning Theory (COLT), 2017, pp.~643--677.

\bibitem{cutkosky2018black}
{\sc A.~Cutkosky and F.~Orabona}, {\em Black-box reductions for parameter-free
  online learning in banach spaces}, in Conference On Learning Theory (COLT),
  2018, pp.~1493--1529.

\bibitem{dekel2015bandit}
{\sc O.~Dekel, R.~Eldan, and T.~Koren}, {\em Bandit smooth convex optimization:
  Improving the bias-variance tradeoff}, in Advances in Neural Information
  Processing Systems, 2015, pp.~2926--2934.

\bibitem{flaxman2005online}
{\sc A.~D. Flaxman, A.~T. Kalai, and H.~B. McMahan}, {\em Online convex
  optimization in the bandit setting: gradient descent without a gradient}, in
  Proceedings of the sixteenth annual ACM-SIAM symposium on Discrete
  algorithms, Society for Industrial and Applied Mathematics, 2005,
  pp.~385--394.

\bibitem{hazan2016optimal}
{\sc E.~Hazan and Y.~Li}, {\em An optimal algorithm for bandit convex
  optimization}, arXiv preprint arXiv:1603.04350,  (2016).

\bibitem{csaba2016}
{\sc L.~P. G.~A. Hu, X. and C.~Szepesv\'{a}ri}, {\em Convex optimization with
  biased noisy gradient oracles}, in International Conference on Artificial
  Intelligence and Statistics (AISTATS), 2016, pp.~819--828.

\bibitem{jun2019parameter}
{\sc K.-S. Jun and F.~Orabona}, {\em Parameter-free online convex optimization
  with sub-exponential noise}, in Conference on Learning Theory (COLT), 2019,
  pp.~1802--1823.

\bibitem{mcmahan2012no}
{\sc B.~Mcmahan and M.~Streeter}, {\em No-regret algorithms for unconstrained
  online convex optimization}, in Advances in neural information processing
  systems, 2012, pp.~2402--2410.

\bibitem{mcmahan2014unconstrained}
{\sc H.~B. McMahan and F.~Orabona}, {\em Unconstrained online linear learning
  in hilbert spaces: Minimax algorithms and normal approximations}, in
  Conference on Learning Theory (COLT), 2014, pp.~1020--1039.

\bibitem{nemirovskii1983problem}
{\sc A.~Nemirovskii, D.~B. Yudin, and E.~R. Dawson}, {\em Problem complexity
  and method efficiency in optimization},  (1983).

\bibitem{NesterovGrFree}
{\sc Y.~Nesterov and V.~Spokoiny}, {\em Random gradient-free minimization of
  convex functions}, Found. Comput. Math., 17 (2017), pp.~527--566.

\bibitem{NH}
{\sc M.~B. Nevelson and R.~Z. Khasminskii}, {\em Stochastic approximation and
  recursive estimation [translated from the Russian by Israel Program for
  Scientific Translations ; translation edited by B. Silver]}, American
  Mathematical Society, 1973.

\bibitem{saha2011improved}
{\sc A.~Saha and A.~Tewari}, {\em Improved regret guarantees for online smooth
  convex optimization with bandit feedback}, in International Conference on
  Artificial Intelligence and Statistics (AISTATS), 2011, pp.~636--642.

\bibitem{shalev2012online}
{\sc S.~Shalev-Shwartz et~al.}, {\em Online learning and online convex
  optimization}, Foundations and Trends{\textregistered} in Machine Learning, 4
  (2012), pp.~107--194.

\bibitem{shamir2013complexity}
{\sc O.~Shamir}, {\em On the complexity of bandit and derivative-free
  stochastic convex optimization}, in Conference on Learning Theory (COLT),
  2013, pp.~3--24.

\bibitem{shamir2017optimal}
{\sc O.~Shamir}, {\em An optimal algorithm for bandit and zero-order convex
  optimization with two-point feedback}, Journal of Machine Learning Research,
  18 (2017), pp.~1--11.

\bibitem{tatiana_ECC2018}
{\sc T.~Tatarenko and M.~Kamgarpour}, {\em Minimizing regret in constrained
  online optimization}, in European Control Conference, 2018.

\bibitem{Thatha}
{\sc A.~L. Thathachar and P.~S. Sastry}, {\em Networks of Learning Automata:
  Techniques for Online Stochastic Optimization}, Springer US, 2003.

\bibitem{yang2016optimistic}
{\sc S.~Yang and M.~Mohri}, {\em Optimistic bandit convex optimization}, in
  Advances in Neural Information Processing Systems, 2016, pp.~2297--2305.

\bibitem{zink}
{\sc M.~Zinkevich}, {\em Online convex programming and generalized
  infinitesimal gradient ascent}, in ICML, Proceedings of Machine Learning
  Research, 2003, pp.~928--936.

\bibitem{zorich}
{\sc V.~Zorich and R.~Cooke}, {\em Mathematical Analysis II}, Mathematical
  Analysis, Springer, 2004.

\end{thebibliography}
\end{document}